\documentclass[a4paper,reqno]{amsart}
\usepackage{array,longtable}
\usepackage{amssymb, amsmath, amscd}
\usepackage[final]{graphicx}
\usepackage{float}\usepackage{wrapfig}
\usepackage{color}
\usepackage{bbm}
\usepackage[final]{graphicx}
\usepackage{tikz-cd}
\usepackage[russian,english]{babel}
\usepackage[utf8]{inputenc}
\usepackage[T1,T2A]{fontenc}

\newtheorem{theorem}{Theorem}[section]
\newtheorem{lemma}[theorem]{Lemma}
\newtheorem{remark}[theorem]{Remark}

\newtheorem{example}[theorem]{Example}

\makeatletter
 \@addtoreset{equation}{section}
\makeatother
\begin{document}
\title[Radial harmonic functions]{Maximal domains of radial harmonic functions}
\author{P. Gilkey}
\email{gilkey@uoregon.edu}
\address{PG: Mathematics Department, University of Oregon, Eugene OR 97405 USA}
\email{gilkey@uoregon.edu}
\author{J. H. Park}
\address{JHP: Department of
Mathematics, Sungkyunkwan University, Suwon, 16419 Korea.}
\email{parkj@skku.edu}
\begin{abstract} We examine the maximal domain of radial harmonic functions on harmonic spaces in
the context of positive, zero, and negative curvature.
\end{abstract}
\maketitle
\section{Introduction}
Throughout this paper, let $\mathbb{M}:=(M,g)$ be a connected
Riemannian manifold of dimension $m\ge4$. Let $\iota_{\mathbb{M},P}$ be
the injectivity radius at $P$,
let $r_{\mathbb{M},P}(Q)$ be the geodesic distance from $P$ to $Q$, and let
$$
B_{\mathbb{M},P}:=\left\{Q\in M:r_{\mathbb{M},P}(Q)<\iota_{\mathbb{M},P}\right\}
$$
be the geodesic ball centered at $P$ of radius $\iota_{\mathbb{M},P}$.
Let $\vec x\rightarrow\exp_P(x^1e_1+\dots+x^me_m)$ define geodesic coordinates centered at $P$
where $\vec e=(e_1,\dots,e_m)$ is an orthonormal frame for $T_PM$ and $\vec x\in\mathbb{R}^m$
satisfies $\|\vec x\|<\iota_{\mathbb{M},P}$.
The geodesic distance from $P$ is then given by the ordinary Euclidean distance, i.e.
$$
r_{\mathbb{M},P}(\vec x)=\|\vec x\|=\sqrt{(x^1)^2+\dots+(x^m)^2}\quad\text{on}\quad B_{\mathbb{M},P}\,.
$$
Let $d\vec x=dx^1\dots dx^m$ be the Euclidean measure
and let $\operatorname{dvol}_{\mathbb{M}}$ be the Riemannian measure on $\mathbb{M}$.
If  $g_{ij}:=g(\partial_{x^i},\partial_{x^j})$, then
$$
\operatorname{dvol}_{\mathbb{M}}=\tilde\Theta_{\mathbb{M},P} dx^1\dots dx^m\quad\text{where}\quad\tilde\Theta_{\mathbb{M},P}:=\det(g_{ij})^{1/2}
$$
is the {\it volume density function}. Let $S_P^{m-1}:=\{\vec\theta\in T_PM:\|\vec\theta\|=1\}$
be the unit sphere in $T_PM$ and let $\mathbb{S}_P^{m-1}:=(S_P^{m-1},g_S)$
where $g_S$ is the induced Euclidean metric.
Introduce geodesic polar coordinates $(r,\vec\theta)$ to express
$$\vec x=r_{\mathbb{M},P}(\vec x)\vec\theta(\vec x)
$$
for $0<r_{\mathbb{M},P}(\vec x):=\|\vec x\|<\iota_{\mathbb{M},P}$ and
$\vec\theta(\vec x)=\|\vec x\|^{-1}\vec x\in S_P^{m-1}$.
We may also express
$$
\operatorname{dvol}_{\mathbb{M}}
=\Theta_{\mathbb{M},P} dr\operatorname{dvol}_{\mathbb{S}_P^{m-1}}\quad\text{for}\quad
\Theta_{\mathbb{M},P}:=r^{m-1}\tilde\Theta_{\mathbb{M},P}\,.
$$

We say that a smooth function $f$, which is defined near $P$, is {\it radial} if
there exists a smooth function $\eta_1$ of one real variable so $f(\vec x)=\eta_1(\|\vec x\|)$;
$f$ is smooth at $P$ if and only if we can write $f(\vec x)=\eta_2(\|x\|^2)$ or, equivalently,
$\eta_1$ is an even function.
We say that $\mathbb{M}$ is {\it central harmonic about $P$}
if $\tilde\Theta_{\mathbb{M},P}$ or,
equivalently, if $\Theta_{\mathbb{M},P}$ is a radial function.
We say that $\mathbb{M}$ is a {\it harmonic space} if $\mathbb{M}$
is central harmonic about every point.

There is a vast literature on this subject; we refer to \cite{B,BTV95,CR40,GP21,N05,R31,R63} and the
references cited therein for further details.
Note that if $\mathbb{M}$ is a harmonic space, then we can rescale the metric to replace $g$ by $c^2g$
for any $c>0$ to obtain another harmonic space $\mathbb{M}_c:=(M,c^2g)$.
Similarly, we showed previously \cite{GP22y} that if $\mathbb{M}$ is central harmonic at
$P$ and if $\psi$ is a smooth positive radial function, then the
radial conformal deformation $\mathbb{M}_\psi:=(M,\psi^2g)$ is again central harmonic at $P$.
Thus we can construct a space which is central harmonic at a single point
by taking a radial conformal deformation
of a harmonic space. There are, however, examples of spaces which are central harmonic at
some point which do not arise in this fashion \cite{GP22y}.

\subsection{Radial harmonic functions}
One can show that $\mathbb{M}$ is central harmonic about $P$ if and only if
there exists a non-constant radial harmonic function $\phi$ with domain
$B_{\mathbb{M},P}-\{P\}$. We will usually assume $\mathbb{M}$ is real analytic;
using analytic hypoellipticity (see, for example, the discussion in Treves~\cite{T71})
this implies $\phi$ is real analytic; we will use this fact implicitly in much
of what follows. In this
paper, we will be concerned with determining a maximal connected domain of $\phi$ as
often $\phi$ has a natural extension beyond $B_{\mathbb{M},P}-\{P\}$.

\subsection{Examples of harmonic spaces}
Let $\rho_{i\ell}:=g^{jk}R_{ijk\ell}$ be the components of the
{\it Ricci tensor} and let $\tau:=g^{i\ell}\rho_{i\ell}$
be the {\it scalar curvature} where we adopt the {\it Einstein convention} and sum over repeated indices.
The rank~1 symmetric spaces play
a central role in the subject; they are all harmonic spaces.
We give below, up to rescaling, the rank~1 symmetric spaces with $\tau>0$
in Section~\ref{S1.4} and the rank~1 symmetric spaces with $\tau<0$ in Section~\ref{S1.6}.
All known examples of complete harmonic spaces are real analytic.

Let $\mathbb{M}$ be a complete simply connected harmonic space;
the scalar curvature $\tau$ is then constant.
We can always rescale to replace $\tau$ by $c^2\tau$ so only the sign of $\tau$
is relevant and it is natural to study
those cases separately.
If $\tau>0$, then $\mathbb{M}$ is positively curved and, after a suitable rescaling
of the metric, the local geometry is modeled
on a positively curved rank 1 symmetric space.
If $\tau=0$, then $\mathbb{M}$ is flat. If $\tau<0$, the known examples are
the negatively curved rank~1 symmetric spaces and the
Damek--Ricci spaces defined in~\cite{DR92}; the Damek--Ricci
spaces are homogeneous negatively curved harmonic spaces
which need not be rank~1 symmetric spaces. The classification of simply connected complete
harmonic spaces is not yet finished in the negatively
curved setting and it is not known if there are additional negatively curved harmonic spaces.

\subsection{Radial harmonic functions}\label{S1.3}
Let $\Delta^0$ be the Laplace-Beltrami operator on functions. If $\mathbb{M}$ is central harmonic at $P$ and if $\phi$ is a radial function, then
$$\Delta^0\phi=-(\partial_r+\Theta_{\mathbb{M},P}^{-1}\dot\Theta_{\mathbb{M},P})\partial_r\phi\,.$$
Let $\phi_1=\Theta_{\mathbb{M},P}^{-1}$ and let
$\phi_0$ solve the ODE $\partial_r\phi_0=\phi_1$. Then
$$
\Delta^0\phi_0=-(\partial_r+\Theta_{\mathbb{M},P}^{-1}\dot\Theta_{\mathbb{M},P})\phi_1=
\Theta_{\mathbb{M},P}^{-2}\dot\Theta_{\mathbb{M},P}-\Theta_{\mathbb{M},P}^{-1}\dot\Theta_{\mathbb{M},P}\Theta_{\mathbb{M},P}^{-1}=0\,.
$$
Consequently, $\phi_0$ is a non-constant radial harmonic function; since
$\Theta_{\mathbb{M},P}$ vanishes
at $P$, $\phi_0$ is singular at $P$. If $\phi$ is any radial solution to
the equation $\Delta^0\phi=0$ on $B_{\mathbb{M},P}-\{P\}$, then we may express
$$
\phi=a\phi_0+b\,.
$$

\subsection{The rank 1 symmetric spaces with positive curvature}\label{S1.4}
Let $\mathbb{S}^m$ be the unit sphere in
$\mathbb{R}^{m+1}$, let $\mathbb{CP}^k$ be
complex projective space, let $\mathbb{HP}^k$ be quaternionic projective space, and let
$\mathbb{OP}^2$ be the Cayley projective plane. We give these spaces the standard metrics
normalized so
\begin{equation}\label{E1.a}
\begin{array}{| l | l | l | l |}\noalign{\hrule}
\mathbb{M}&\text{dimension}&\text{diameter}&\Theta_{\mathbb{M},P}\\\noalign{\hrule}
\mathbb{S}^m&m&\pi&\sin(r)^{m-1}\vphantom{\vrule height 10pt}\\\noalign{\hrule}
\mathbb{CP}^k&2k&\frac12\pi&\sin(r)^{2k-1}\cos(r)\vphantom{A_{\vrule height 9pt}^{\vrule height 7pt}}\\\noalign{\hrule}
\mathbb{HP}^k&4k&\frac12\pi&\sin(r)^{4k-1}\cos^3(r)\vphantom{A_{\vrule height 9pt}^{\vrule height 7pt}}\\\noalign{\hrule}
\mathbb{OP}^2&16&\frac12\pi&\sin(r)^{15}\cos^7(r)\vphantom{A_{\vrule height 9pt}^{\vrule height 7pt}}\\\noalign{\hrule}
\end{array}\end{equation}
The metric on $\mathbb{S}^m$ is the standard metric inherited from Euclidean space,
the metric on $\mathbb{CP}^k$ is the suitably normalized Fubini-Study metric, and
so forth. The rank 1 symmetric spaces in positive curvature
are compact 2 point homogeneous spaces with
$B_{\mathbb{M},P}=M-\mathcal{C}_{\mathbb{M},P}$ where
$\mathcal{C}_{\mathbb{M},P}$ is the cut-locus:
$$\mathcal{C}_{\mathbb{S}^m,P}=\{-P\},\quad
\mathcal{C}_{\mathbb{CP}^k,P}=\mathbb{CP}^{k-1},\quad
\mathcal{C}_{\mathbb{HP}^k,P}=\mathbb{HP}^{k-1},\quad
\mathcal{C}_{\mathbb{OP}^2,P}=S^7\,.
$$

Let $\mathbb{M}$ be a complete connected harmonic space of positive curvature which is not simply
connected. Let $\mathbb{M}_1$ be the universal cover of $\mathbb{M}$. Then $\mathbb{M}_1$
is a positively curved rank 1 symmetric space.
We will establish the following result in Section~\ref{S2} which
applies to the harmonic spaces of positive curvature but is more general since
no assumption on the metric is made. Although it
is well known, we shall present a proof in Section~\ref{S2} based on the Gauss-Bonnet
formula (see Theorem~\ref{T2.2})
as it motivates a number of examples that we will discuss in Section~\ref{S2}.

\begin{theorem}\label{T1.1}\rm Let $\tilde{\mathbb{M}}=(\tilde M,\tilde g)$ be the universal cover of a smooth
connected complete Riemannian manifold
$\mathbb{M}=(M,g)$ which is not simply-connected.
\begin{enumerate}
\item If $\tilde{\mathbb{M}}=\mathbb{S}^m$ and if $m$ is even, then $\pi_1(M)=\mathbb{Z}_2$ and $M$ is not orientable.
\item If $\tilde{\mathbb{M}}\in\{\mathbb{CP}^k,\mathbb{HP}^k\}$, then $k$ is odd, $\pi_1(M)=\mathbb{Z}_2$, and $M$ is not orientable.
\item $\tilde{\mathbb{M}}$ is not $\mathbb{OP}^2$.
\end{enumerate}\end{theorem}

\subsection{The maximal domain of $\phi_0$ in non-positive curvature}\label{S1.5}
Let $\mathbb{M}$ be a complete connected Riemannian manifold and let
$\pi:\tilde{\mathbb{M}}\rightarrow\mathbb{M}$ be the universal cover; we assume
$\mathbb{M}$ is not simply connected so $\pi$ is nontrivial. We assume that
$\tilde{\mathbb{M}}$ is a non-positively curved harmonic space. Since $\mathbb{M}$ is
non-positively curved, the exponential map is a covering projection so we can identify
$\tilde M$ with $T_PM$ and $\pi$ with $\exp_P$ for any point $P$ of $M$. There are
no conjugate points in $\mathbb{M}$ or $\tilde{\mathbb{M}}$; there is a unique geodesic
segment minimizing distance between any two points of $\tilde{\mathbb{M}}$ but there
can be several such geodesic segments in $\mathbb{M}$.
Set $\Gamma:=\pi_1(M)$. Then
$\Gamma$ acts on $\tilde{\mathbb{M}}$ by isometries and we may identify
$\mathbb{M}=\tilde{\mathbb{M}}/\Gamma$. If $\tilde P\in\tilde{\mathbb{M}}$, then
$\Gamma\cdot\tilde P$ is a discrete subset of $\tilde{\mathbb{M}}$ so we can replace $\inf$
by $\min$ to define
$$
\tilde{\mathcal{O}}_{\tilde P}:
=\left\{\tilde Q\in\tilde M:d_{\tilde{\mathbb{M}}}(\tilde P,\tilde Q)
<\min_{\operatorname{id}\ne\gamma\in\Gamma}d_{\tilde{\mathbb{M}}}
(\tilde P,\gamma\tilde Q)\right\}\,.
$$
Since $\tilde{\mathcal{O}}_{\gamma\tilde P}
=\gamma\tilde{\mathcal{O}}_{\tilde P}$, $\pi(\tilde{\mathcal{O}}_{\tilde P})$ is independent
of the particular point in $\pi^{-1}\{P\}$ which is chosen and we may define
$$
\mathcal{O}_P:=\pi(\tilde{\mathcal{O}}_{\tilde P})\quad\text{for any}\quad\tilde P\in\pi^{-1}\{P\}\,.
$$
Because we can use $\min$ instead of $\inf$ in defining $\tilde{\mathcal{O}}_{\tilde P}$,
$\tilde{\mathcal{O}}_{\tilde P}$ is an open subset of $\tilde{\mathbb{M}}$
and $\pi:\tilde{\mathcal{O}_P}\rightarrow
\mathcal{O}_P$ is a 1-1, onto, isometric map. Clearly
$\mathcal{O}_P$ is the set of all points in $\mathbb{M}$
so there is a unique geodesic segment from $P$ to the point in question minimizing the distance.
If we replace strict inequality by equality, we obtain
$$
\operatorname{bd}(\tilde{\mathcal{O}}_{\tilde P})=
\left\{\tilde Q\in\tilde M:d_{\tilde{\mathbb{M}}}(\tilde P,\tilde Q)
=\min_{\operatorname{id}\ne\gamma\in\Gamma}d_{\tilde{\mathbb{M}}}
(\tilde P,\gamma\tilde Q)\right\}\,.
$$
Since the cut locus of $\mathbb{M}$ is the set of all points where there are 2 distinct
minimizing geodesic segments from $P$ to the point in question, we obtain
$$
\operatorname{bd}(\mathcal{O}_P)=\pi\left\{\operatorname{bd}(\tilde{\mathcal{O}}_{\tilde P})\right\}=
\mathcal{C}_P\,.
$$
Since the cut-locus has measure 0 in $\mathbb{M}$ and since $M=\overline{\mathcal{O}_P}$,
we obtain
\begin{equation}\label{E1.b}
\operatorname{vol}(\mathbb{M})=\operatorname{vol}(\mathcal{O}_P)\,.
\end{equation}
It is a straightforward to see that
\begin{equation}\label{E1.c}
\iota_{\mathbb{M},P}=\frac12\min_{\operatorname{id}\ne\gamma\in\Gamma}
d_{\tilde{\mathbb{M}}}(\tilde P,\gamma\tilde P)\text{ for any }\tilde P\in\pi^{-1}\{P\}\,.
\end{equation}

Note that $B_{\mathbb{M},P}\subset\mathcal{O}_P$ and in general $\mathcal{O}_P$ is a bigger
open set. Let $\phi_{\tilde{\mathbb{M}}}$ be a non-constant harmonic function on
$\tilde{\mathbb{M}}-\{\tilde P\}$ which is radial from $\tilde P$. Let
$$
\phi_{\mathbb{M}}:=\phi_{\tilde{\mathbb{M}}}\circ\left\{\pi|_{\tilde{\mathcal{O}}_{\tilde P}}\right\}^{-1}\,.
$$
The following result shows that $\mathcal{O}_P=M-\mathcal{C}_P$ is
the maximal domain in $\mathbb{M}$ which admits a non-constant harmonic function
which is radial from $P$.

\begin{theorem}\label{T1.2}\rm Let $\mathbb{M}$ be a non-positively curved complete connected
harmonic space. Adopt the notation established above.
\begin{enumerate}
\item $\phi_{\mathbb{M}}$ is a non-constant function  on $\mathcal{O}_P-\{P\}$
which is smooth  harmonic, and radial from $P$.
\item If $\mathcal{O}$ is an open set in $M$ which contains a point of $\mathcal{C}_P$, then $\mathcal{O}$ does not admit a smooth  non-constant harmonic function which is radial from $P$.\end{enumerate}\end{theorem}

\begin{proof} Since $\pi$ is a local isometry, $\phi_{\mathbb{M}}$ is harmonic. Since $r_{\tilde{\mathbb{M}},\tilde P}(\tilde Q)=
r_{\mathbb{M},\pi(P)}(\pi(Q))$ for $\tilde Q\in\tilde{\mathcal{O}}_{\tilde P}$ and since $\tilde\phi$
is radial from $\tilde P$ in $\tilde{\mathbb{M}}$, we see that $\phi$ is a radial function from $P$
on the open set $\mathcal{O}_P$. This proves Assertion~(1).
Since $r_{\mathbb{M},P}$ is Lipschitz but not smooth on the
cut-locus. Assertion~(2) follows.
\end{proof}

There are other open sets to which $\phi_0$ can be extended as a smooth harmonic
function but $\phi_0$ will no longer be radial there; we refer to Remark~\ref{R3.1} for details.

\subsection{The rank 1 symmetric spaces of negative curvature}\label{S1.6}
There are negative curvature duals of the spaces discussed in Section~\ref{S1.4} that we shall
denote by $\widetilde{\mathbb{S}}^m$ (hyperbolic space),
$\widetilde{\mathbb{CP}^k}$ {(complex hyperbolic space)},
$\widetilde{\mathbb{HP}}^k$ {(quaternionic hyperbolic  space)}, and
$\widetilde{\mathbb{OP}}^2$ {(Cayley hyperbolic plane)}. These are the rank 1 symmetric
spaces of negative curvature; they are all 2-point homogeneous spaces and are
geodesically complete. The curvature tensor of these spaces is obtained by reversing
the sign of the curvature tensor of the corresponding positive curvature example. This fact
will play an important role in the discussion of Section~\ref{S4}. Finally, we note that
any simply-connected 2-point homogeneous space
is either flat or is a rank 1 symmetric space.

If $\mathbb{M}$ is a rank 1 symmetric space with negative curvature, then
the exponential map is a global diffeomorphism so the underlying topology of all these spaces
is Euclidean space; the cut locus is empty.
We adopt the same normalizations as those used to normalize the positive curvature examples.
We replace $\sin$ by $\sinh$ and $\cos$ by $\cosh$ in Equation~(\ref{E1.a}) to obtain
$$
\begin{array}{| l | l | l |}\noalign{\hrule}
\mathbb{M}&\text{d{imension}}&\Theta_{\mathbb{M},P}\\\noalign{\hrule}
\widetilde{\mathbb{S}}^m&m&\sinh(r)^{m-1}
\vphantom{\vrule height 12pt}\\\noalign{\hrule}
\widetilde{\mathbb{CP}}^k&2k&\sinh(r)^{2k-1}\cosh(r)
\vphantom{A_{\vrule height 9pt}^{\vrule height 7pt}}\\\noalign{\hrule}
\widetilde{\mathbb{HP}}^k&4k&\sinh(r)^{4k-1}\cosh^3(r)
\vphantom{A_{\vrule height 9pt}^{\vrule height 7pt}}\\\noalign{\hrule}
\widetilde{\mathbb{OP}}^2&16&\sinh(r)^{15}\cosh^7(r)
\vphantom{A_{\vrule height 9pt}^{\vrule height 7pt}}\\\noalign{\hrule}
\end{array}$$

{We say that $\mathbb{M}$ is modeled on a homogeneous space $\mathbb{M}_1$
if every point of $\mathbb{M}$ has a neighborhood which is isometric to some open set
in $\mathbb{M}_1$.} We will
establish the following result in Section~\ref{S4.2} which gives lower bounds
on the volumes of compact spaces modeled on rank 1 symmetric spaces. It is not
difficult to show that $\mathcal{O}_P-\{P\}$ is of full measure in $M$ and, consequently,
by Theorem~\ref{T1.2} and Equation~(\ref{E1.b}), we obtain on the volume of the maximal domain of $\phi_0$. We refer to
Cahn et. al. \cite{CGW76} for a more general discussion of lower volume bounds
on even dimensional negatively curved
symmetric spaces.
\begin{theorem}\label{T1.4}\rm Let $\mathbb{M}$ be a compact Riemannian manifold.
\begin{enumerate}
\item If the geometry of $\mathbb{M}$ is modeled on ${\widetilde{\mathbb{S}}}^{2k}$, then
$\operatorname{vol}(\mathbb{M})\ge\frac12\operatorname{vol}(\mathbb{S}^{2k})$.
\item If the geometry of $\mathbb{M}$ is modeled on ${\widetilde{\mathbb{CP}}}^k$, then
$\operatorname{vol}(\mathbb{M})\ge\frac1{k+1}\operatorname{vol}(\mathbb{CP}^k)$.
\item If the geometry of $\mathbb{M}$ is modeled on ${\widetilde{\mathbb{HP}}}^k$, then
$\operatorname{vol}(\mathbb{M})\ge\frac1{k+1}\operatorname{vol}(\mathbb{HP}^k)$.
\item If the geometry of $\mathbb{M}$ is modeled on ${\widetilde{\mathbb{OP}}}^2$, then
$\operatorname{vol}(\mathbb{M})\ge\frac1{3}\operatorname{vol}(\mathbb{HP}^k)$.
\end{enumerate}
\end{theorem}

The estimates of Theorem~\ref{T1.4} will
arise from the Chern-Gauss-Bonnet Formula (see Theorem~\ref{T2.2}) and are not
optimal in certain settings.
We will use the Hirzebruch Signature Formula (see Theorem~\ref{T2.2})
to establish the following result in Section~\ref{S4.3}.

\begin{theorem}\label{T1.5}\rm Let $\mathbb{M}$ be a compact Riemannian manifold. Let
$$
\varepsilon(M):=\left\{\begin{array}{lll}1&\text{if }M\text{ is orientable}\\
\frac12&\text{if }M\text{ is not orientable}\end{array}\right\}\,.
$$
\begin{enumerate}
\item If the geometry of $\mathbb{M}$ is modeled on $\widetilde{\mathbb{CP}}^{2k}$, then
$\operatorname{vol}(\mathbb{M})\ge\varepsilon(M)\operatorname{vol}(\mathbb{CP}^{2k})$.
\item If the geometry of $\mathbb{M}$ is modeled on $\widetilde{\mathbb{HP}}^{2k}$, then
$\operatorname{vol}(\mathbb{M})\ge\varepsilon(M)\operatorname{vol}(\mathbb{HP}^{2k})$.
\item If the geometry of $\mathbb{M}$ is modeled on $\widetilde{\mathbb{OP}}^2$, then
$\operatorname{vol}(\mathbb{M})\ge\varepsilon(M)\operatorname{vol}(\mathbb{OP}^2)$.
\end{enumerate}\end{theorem}

\subsection{Outline of the paper}
In Section~\ref{S2}, we discuss the positively curved rank~1 symmetric spaces.
We exhibit quite explicitly the non-constant harmonic function
$\phi_0$ on the punctured disk $B_{\mathbb{M},P}-\{P\}$
for some low dimensional examples and complete the proof of Theorem~\ref{T1.1}. We  discuss in
some detail the domain of $\phi_0$ on a homogeneous lens space
$\mathbb{S}^{2k+1}/\mathbb{Z}_4$, and a quotient
$\mathbb{CP}^{2k}/\mathbb{Z}_2$.
In Section~\ref{S3}, we turn our attention to flat space and discuss the rectangular torus
and the infinite mobius strip.
In Section~\ref{S4}, we exhibit the non-constant harmonic function $\phi_0$ in low dimensions for the
negatively curved symmetric spaces and we complete the proof of Theorems~\ref{T1.4} and \ref{T1.5}.

The basepoint $P$ of the manifold will be fixed for the most part.
To simplify the notation, we shall suppress the dependence in the notation on the point $P$
and set $r_{\mathbb{M}}:=r_{\mathbb{M},P}$, $\Theta_{\mathbb{M}}:=\Theta_{\mathbb{M},P}$,
$\iota_{\mathbb{M}}:=\iota_{\mathbb{M},P}$ and so forth when
no confusion is likely to ensue.

\section{Positive curvature}\label{S2}
Let $\mathbb{M}$ be central harmonic at $P$. In Section~\ref{S1.3}, we discussed the construction of
a non-constant radial harmonic function $\phi_0$ on $B_{\mathbb{M},P}$. In Section~\ref{S2.1}, we exhibit
this function for the rank 1 symmetric spaces of positive curvature in low dimensions. In Section~\ref{S2.2},
we turn our attention to the study of non-simply connected examples $\mathbb{M}/\Gamma$ where $\mathbb{M}$
is a positively curved rank 1 symmetric space and $\Gamma$ is a finite group of isometries acting without
fixed points to
complete the proof of Theorem~\ref{T1.1}. In Section~\ref{S2.3}, we discuss spherical space forms and construct
a homogeneous example $\mathbb{S}^{2k+1}/\mathbb{Z}_4$ for which we give the maximal domain of
$\phi_0$. In
Section~\ref{S2.4} give an
isometric fixed point free action of $\mathbb{Z}_2$ on $\mathbb{CP}^{2k+1}$ and discuss the maximal domain
of $\phi_0$ on $\mathbb{CP}^{2k+1}/\mathbb{Z}_2$.

\subsection{Simply-connected examples}\label{S2.1}
We use Mathematica to determine $\phi_1$ and $\phi_0$ in low dimensions for
the rank 1 symmetric spaces of positive curvature.
\goodbreak\smallbreak\vbox{\centerline{Table 2.1}
\smallbreak\centerline{$\begin{array}{| l | l | l |}\noalign{\hrule}\mathbb{M}&\phi_1&\phi_0\\\noalign{\hrule}
\mathbb{S}^2&\frac1{\sin(r)}&\log \left(\tan \left(\frac{r}{2}\right)\right)\\\noalign{\hrule}
\mathbb{S}^3&\frac1{\sin(r)^2}&-\cot (r)\\\noalign{\hrule}
\mathbb{S}^4&\frac1{\sin(r)^3}&-\frac{1}{8} \csc ^2\left(\frac{r}{2}\right)+\frac{1}{8} \sec ^2\left(\frac{r}{2}\right)+\frac{1}{2} \log \left(\tan \left(\frac{r}{2}\right)\right)\\\noalign{\hrule}
\mathbb{S}^5&\frac1{\sin(r)^4}&-\frac23 \cot (r)-\frac{1}{3} \cot (r) \csc ^2(r)\\
\noalign{\hrule}
\mathbb{CP}^2&\frac1{\sin(r)^3\cos(r)}&-\frac{1}{2} \csc ^2(r)+\log (\tan (r)) \\\noalign{\hrule}
\mathbb{CP}^3&\frac1{\sin(r)^5\cos(r)}&-\frac{1}{4} \csc ^4(r)-\frac12\csc ^2(r)+\log (\tan (r)) \\ \noalign{\hrule}
\mathbb{CP}^4&\frac1{\sin(r)^7\cos(r)}&-\frac{1}{6} \csc ^6(r)-\frac14\csc ^4(r)-\frac12\csc ^2(r)+\log (\tan (r))\\
\noalign{\hrule}
\mathbb{HP}^2&\frac1{\sin(r)^7\cos(r)^3}&\frac{1}{2} \left(-\frac{1}{3} \csc ^6(r)-\csc ^4(r)-3 \csc ^2(r)+\sec ^2(r)\right.\\
&&\phantom{-\frac13}\left.+8 \log (\tan (r))\right)\\ \noalign{\hrule}
\mathbb{HP}^3&\frac1{\sin(r)^{11}\cos(r)^3}&\frac{1}{2} \left(-\frac{1}{5} \csc ^{10}(r)-\frac12\csc ^8(r)
-\csc ^6(r)-2 \csc ^4(r)\right.\\
&&\phantom{\frac12}\left.-5 \csc ^2(r)+\sec ^2(r)+12 \log (\tan (r))\right)\\\noalign{\hrule}
\mathbb{HP}^4&\frac1{\sin(r)^{15}\cos(r)^3}&\frac{1}{2} \left(-\frac{1}{7} \csc ^{14}(r)-\frac13\csc ^{12}(r)
-\frac35 \csc ^{10}(r)-\csc ^8(r)\right.\\
&&\phantom{\frac12}-\frac53 \csc ^6(r)-3 \csc ^4(r)-7 \csc ^2(r)+\sec ^2(r)\\
&&\phantom{\frac12}\left.+16 \log (\tan (r))\vphantom{\frac13}\right)\goodbreak\\\noalign{\hrule}\noalign{\goodbreak\hrule}
\mathbb{OP}^2&\frac1{\sin(r)^{15}\cos(r)^7}&
-\frac{1}{14} \csc ^{14}(r)-\frac13\csc ^{12}(r)-\csc ^{10}(r)-\frac52 \csc ^8(r)\\
&&-\frac{35}6 \csc ^6(r)-14 \csc ^4(r)-42 \csc ^2(r)+\frac16\sec ^6(r)\\
&&+2 \sec ^4(r)+18 \sec ^2(r)+120 \log (\tan (r))
\\ \noalign{\hrule}\end{array}$}}
\goodbreak\medbreak\noindent
For the sphere, since $\sin(r)^{1-m}\sim(\pi-r)^{1-m}$ as $r\rightarrow\pi$,
$\phi_1$ is not integrable on $(\pi-\varepsilon,\pi)$. Consequently, $\phi_0$ does not extend to the antipode and $B_{\mathbb{M},P}-\{P\}$ is the natural domain of definition.
For the remaining rank 1 symmetric spaces, since $\phi_1=\sin(\pi-r)^{m-1}\cos(\pi-r)^{-k}$
is not integrable on $(\frac12\pi-\varepsilon,\frac12\pi)$ for $k>0$, $\phi_0$ tends to infinity as $r\rightarrow\frac12\pi$ and again the natural
domain of $\phi_0$ is $B_{\mathbb{M},P}-\{P\}=\mathbb{M}-\mathcal{C}$.

\subsection{Non simply-connected examples: The proof of Theorem~\ref{T1.1}}\label{S2.2}
Theorem~\ref{T1.1} follows from Theorem~9.3.1 of Wolf~\cite{W2010}.
However, to keep this discussion
as self-contained as possible and to provide some examples, we shall give a different
discussion based on the Chern-Gauss-Bonnet Formula for the most part; an exception being
the case of $\mathbb{HP}^{2\ell+1}$ for $\ell>1$ where we do not know an elementary proof
that $\mathbb{HP}^{2\ell+1}$ does not admit a finite group acting without fixed points.
We note that
since we are in positive curvature,
Synge's Theorem is relevant
although it does not give the full result.

Let $\mathbb{M}$ be a compact Riemannian manifold
of dimension $m$.
If $m=2j$, let $\chi(M)$ be the Euler characteristic of $M$ and,
if $m=4k$ and if $M$ is orientable, let $\operatorname{sign}(M)$
be the signature of $M$. We recall the following well known result:
\begin{lemma}\label{L2.1}\rm
\ \begin{enumerate}
\item The cohomology rings of ${\rm CP}^k$, ${\rm HP}^k$, and ${\rm OP}^2$ are truncated
polynomial rings on generators $x_i$ of degree $i$:
\begin{eqnarray*}
&&H^*({\rm CP}^k)=\mathbb{R}[x_2]/(x_2^{k+1}=0)
=\mathbb{R}\oplus\mathbb{R}\cdot x_2\oplus\dots\dots\dots\oplus\mathbb{R}\cdot x_2^k,\\
&&H^*({\rm HP}^k)=\mathbb{R}[x_4]/(x_4^{k+1}=0)
=\mathbb{R}\oplus\mathbb{R}\cdot x_4\oplus\dots\dots\dots
\oplus\mathbb{R}\cdot x_4^k,\\
&&H^*({\rm OP}^2)=\mathbb{R}[x_8]/(x_8^3=0)
=\mathbb{R}\oplus\mathbb{R}\cdot x_8\oplus\mathbb{R}\cdot x_8^2\,.
\end{eqnarray*}
\item The Euler characteristic and the Hirzebruch signature are given by:
$$
\begin{array}{| l | c  | c | l | c | c |}\noalign{\hrule}
\mathbb{M}&\chi(M)&\operatorname{sign}(M)&
\mathbb{M}&\chi(M)&\operatorname{sign}(M)\\\noalign{\hrule}
\mathbb{S}^{4k+2}&2&-&\mathbb{S}^{4k}&2&0\\ \noalign{\hrule}
\mathbb{CP}^{2k+1}&2k+2&-&\mathbb{CP}^{2k}&2k+1&1\\ \noalign{\hrule}
\mathbb{HP}^{2k+1}&2k+2&0&\mathbb{HP}^{2k}&2k+1&1\\ \noalign{\hrule}
\mathbb{OP}^2&3&1& & &\\ \noalign{\hrule}
\end{array}
$$
\end{enumerate}\end{lemma}

Let $\operatorname{Pf}_j$ be the Pfaffian and let $L_k$ be the Hirzebruch
polynomial in the curvature of $\mathbb{M}$. We refer to Chern~\cite{C44} for the proof of Assertion~(1)
and to Hirzebruch~\cite{H78} for the proof of Assertion~(2) in the following result.

\begin{theorem}\label{T2.2}\rm\ Let $\mathbb{M}$ be a compact Riemannian manifold
of dimension $m$.
Let $\pi:\mathbb{M}_1\rightarrow\mathbb{M}$ be a finite $\ell$-fold Riemannian cover.
\begin{enumerate}
\item If $m=2k$, then
$\chi(M)=\int_M\operatorname{Pf}_k(\mathbb{M})$ and thus $\chi(M_1)=\ell\cdot\chi(M)$.
\item If $m=4k$ and if $M$ is orientable, then
$\operatorname{sign}(M)=\int_ML_k(\mathbb{M})$ and thus
$\operatorname{sign}(M_1)=\ell\cdot\operatorname{sign}(M)$.
\end{enumerate}\end{theorem}

Theorem~\ref{T1.1} will follow from the following result.

\begin{lemma}\label{L2.3}\rm\
Let $\Gamma$ be a finite group which acts without fixed points on a compact
manifold $M$ of dimension $m=2\ell\cdot n$. Assume that
\begin{equation}\label{E2.a}
H^*(M)=\mathbb{R}[x_{2\ell}]/(x_{2\ell}^{n+1}=0)=\mathbb{R}\oplus\mathbb{R}x_{2\ell}\oplus\dots\oplus
\mathbb{R}x_{2\ell}^n
\end{equation}
is a truncated polynomial ring where $x_{2\ell}\in H^{2\ell}(M)$.
\begin{enumerate}
\item If $n$ is even, then $\Gamma$ is trivial.
\item If $n$ is odd, then either $\Gamma$ is trivial or $\Gamma=\mathbb{Z}_2$ and $\mathbb{M}/\Gamma$ is not
orientable.
\item If $M\in\{\mathbb{CP}^{2k},\mathbb{HP}^{2k}\}$ for $k\ge1$
or $M=\mathbb{OP}^2$,
then $\Gamma$ is trivial.
\item If $M\in\{\mathbb{CP}^{2k+1},\mathbb{HP}^{2k+1}\}$, then either $\Gamma$
trivial or  $\Gamma=\mathbb{Z}_2$ and $M/\Gamma$ is not orientable.
\end{enumerate}\end{lemma}

\begin{proof} By averaging a Riemannian metric on $M$ over the group $\Gamma$, we may assume without loss
of generality that $\Gamma$ acts by isometries. Since $x_{2\ell}^j$ has even degree for all $j$,
Equation~(\ref{E2.a}) yields $\chi(M)=n+1$.
By Theorem~\ref{T2.2}, $\chi(M)=|\Gamma|\cdot\chi(M/\Gamma)$
and consequently $|\Gamma|$ divides $n+1$.
Let $T\in\Gamma$. We may express $T^*x_{2\ell}=\varepsilon(T)x_{2\ell}$ where $\varepsilon(T)\in\mathbb{R}$.
Since $T^{|\Gamma|}=\operatorname{id}$ and since $|\Gamma|$ divides $n+1$,
$T^{n+1}=\operatorname{id}$. We show $\varepsilon(T)=\pm1$ by computing:
$$
x_{2\ell}=T^*(\operatorname{id})x_{2\ell}=(T^*)^{n+1}(x_{2\ell})=\varepsilon^{n+1}(T)x_{2\ell}\,.
$$

Suppose that $n$ is even. Since $\varepsilon(T)^{n+1}=1$,
this implies $\varepsilon(T)=1$ and consequently $T^*x_{2\ell}^j=x_{2\ell}^j$ for all $j$.
Since the cohomology of $M/\Gamma$ can be identified
with the $\Gamma$-invariant cohomology of $M$, $H^*(M/\Gamma)=H^*(M)$ and consequently
$\chi(M/\Gamma)=\chi(M)$. This implies $\Gamma$ is trivial and completes the proof of Assertion~(1).

Suppose that $n$ is odd and $\Gamma$ is non-trivial.
The map $T\rightarrow\varepsilon(T)$ is a group homomorphism from $\Gamma$ to
$\mathbb{Z}_2=\{\pm1\}$. Let $\Gamma_0=\ker(\varepsilon)$ and let $M_0:=M/\Gamma_0$. The same argument
given to prove Assertion~(1) now shows $M_0=M$ and hence $\ker(\varepsilon)$ is trivial. This
implies $\Gamma=\mathbb{Z}_2$ and that if $T\ne\operatorname{id}$, $T^*(x_{2\ell})=-x_{2\ell}$. We then have
$T^*(x_{2\ell}^n)=-x_{2\ell}^n$ so $T$ reverses the orientation and $M/\Gamma$ is not orientable.
This completes the proof of
Assertion~(2). Assertions~(3) and (4) then follow from Assertion~(1) and Lemma~\ref{L2.1}.
\end{proof}

\subsection{Spherical space forms}\label{S2.3} We shall say that
$\mathbb{M}$ is a spherical space form if $\mathbb{M}$ has constant sectional
curvature $+1$ or, equivalently, if the universal cover of $\mathbb{M}$ is the sphere
$\mathbb{S}^m$ for $m\ge2$.

Let $\Gamma\subset O(m+1)$ be a non-trivial finite subgroup of the orthogonal group
acting by isometries on $\mathbb{S}^m$ without
fixedpoints. If $m$ is even, then $\chi({{{\mathbb{S}}}}^m)=2$ so $|\Gamma|$ divides $2$ by
Theorem~\ref{T2.2} and hence $|\Gamma|=2$
since $\Gamma$ is assumed non-trivial. Let $\gamma$ be the non-trivial element of $\Gamma$. Since $\gamma^2=\operatorname{id}$ and
$\gamma$ is fixed point free, $\gamma=-\operatorname{id}$ and $\mathbb{M}=\mathbb{RP}^m$.
This is a homogeneous space, so the point in question is irrelevant.
The diameter of {{$\mathbb{RP}^m$}} is
$\frac\pi2$, the cut-locus is $\mathbb{RP}^{m-1}$, $\iota_{\mathbb{RP}^m,P}=\frac\pi2$, and
$B_{\mathbb{M},P}$ is the maximal domain of definition for $\phi_0$.

If $m=2n+1$, the situation is very different. The possible groups and group actions which can arise have been classified by Wolf~\cite{W2010}.
There is one example which will play an important role in our future development.
\begin{example}\label{Ex2.4}\rm
If we identify $\mathbb{R}^{2k+2}=\mathbb{C}^{k+1}$
and $\mathbb{Z}_4=\{\pm1,\pm i\}$, then complex multiplication defines a fixed point free isometric action of $\mathbb{Z}_4$ on
${\mathbb{S}}^{2k+1}$ and we let {{$\mathbb{M}=\mathbb{S}^{2k+1}/\mathbb{Z}_4$}}.
Alternatively, let $T$ generate $\mathbb{Z}_4$.
We can define the action in purely real coordinates by setting
\begin{equation}\label{E2.b}
T(x_1,x_2,\dots,x_{2k+2})=(-x_2,x_1,\dots,-x_{2k+2},x_{2k+1})\,.
\end{equation}
We then have a sequence of 2-fold {Riemannian} covering projections
${{\mathbb{S}^{2k+1}}}\rightarrow\mathbb{RP}^{2k+1}$ and
$\mathbb{RP}^{2k+1}\rightarrow\mathbb{M}$ so we can equally
well regard $\mathbb{M}=\mathbb{RP}^{2k+1}/\mathbb{Z}_2$.
The unitary group acts transitively on $\mathbb{S}^{2k+1}$ {by isometries} and commutes with
this action of $\mathbb{Z}_4$ so $\mathbb{RP}^{2k+1}$ and $\mathbb{M}$
are homogeneous spaces. We take as basepoint
$P=(1,0,\dots,0)$ in $\mathbb{R}^{2k+2}$. Let $\xi$ be a unit vector
which is perpendicular to $P$.
$\sigma(t):=\cos(t)P+\sin(t)\xi$ is a unit speed
geodesic from $P$ with $\dot\sigma(0)=\xi$. It now follows that
\begin{eqnarray*}
&&r_{\mathbb{S}^{2k+1},P}(x^1,\dots,x^{2k+2})=\arccos(x^1)\quad\text{so}\\
&&B_{\mathbb{M},P}=\{\vec x\in{S}^m:x^1>x^2\}\cap\{\vec x\in{S}^m:x^1>-x^2\}\,;
\end{eqnarray*}
the condition $x^1>0$ imposed by $T^2$ is then immediate. This is the intersection of the two hemispheres centered at $(\frac1{\sqrt2},\pm\frac1{\sqrt2},0,\dots,0)$;
$\iota_{\mathbb{M},P}=\arccos(\frac1{\sqrt2})=\frac\pi4$. We present a picture below (where the $x^1$ axis is vertical)
of what results when we take a slice by
setting $x^4=\dots=x^{2k+2}=0$; this picture motivates the terminology {\it lens spaces}.
The $90$ degree rotation in the $(x^1,x^2)$ plane
identifies one great circle in boundary of the fundamental domain with the other. The fundamental domain is in light blue; the ball of radius
$\frac\pi4$ about the north pole is in dark blue.
\smallbreak\vbox{\centerline{Picture 2.1}\smallbreak\centerline
{\includegraphics[height=3cm,keepaspectratio=true]{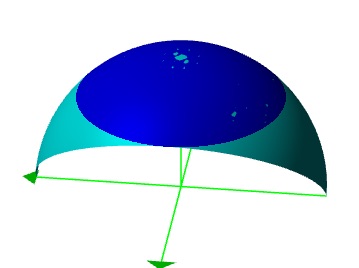}}}
\end{example}

\begin{remark}\rm Let $\mathbb{M}=\mathbb{S}^{2k+1}/\Gamma$
where $\Gamma$ is a finite subgroup of the orthogonal group $O(2k+2)$.
If $\mathbb{M}$ is homogeneous, then, up to conjugation, either $\pi_1(\mathbb{M})$ is the group of units
of order $q$ in $\mathbb{C}$
acting on $\mathbb{S}^{2k+1}\subset\mathbb{C}^{k+1}$ by complex multiplication
 or $m\equiv3\mod4$ and the group in question is one of 4 exceptional groups;
the spherical space forms are rarely homogeneous. We refer to Theorem 11.6.1 of Wolf~\cite{W61}
for details. Similarly, we refer to Wolf~\cite{W61} Corollary 10.6 to
see that if $\mathbb{M}$ is a connected homogeneous manifold of constant negative sectional
curvature, then $\mathbb{M}$ is hyperbolic space. Thus the examples
$\widetilde{\mathbb{S}}^m/\Gamma$ are never homogeneous.
\end{remark}

\subsection{An isometric fixed point free action of $\mathbb{Z}_2$ on $\mathbb{CP}^{2k+1}$}\label{S2.4}Lemma~\ref{L2.3} is a theorem in topology; no assumption on the metric is assumed. It leaves open
the question of finding fixed point free actions of $\mathbb{Z}_2$ on $\mathbb{CP}^{2k+1}$ and
$\mathbb{HP}^{2k+1}$ for $k\ge1$. In the following extended example, we construct an isometric
fixed point free action of $\mathbb{Z}_2$ on $\mathbb{CP}^{2k+1}$ by generalizing the construction
of Example~\ref{Ex2.4}. We note that Theorem 9.3.1 of Wolf~\cite{W2010} shows that
there is no fixed point
free action isometric action of $\mathbb{Z}_2$ on $\mathbb{HP}^{2k+1}$ for $k\ge1$; we know of no elementary proof
of this fact.

Let $\vec z=(z^1,\dots,z^{2k+2})\in S^{4k+3}\subset\mathbb{C}^{2k+2}$ and
let $\langle\vec z\rangle$ denote homogeneous coordinates on $\mathbb{CP}^{2k+1}$ where we identify
$\langle\vec z\rangle$ with
$\langle \mu\vec z\rangle$ for $\mu\in {{{S}^1}}\subset\mathbb{C}$.
We can construct a fixed point free isometry
$T$ of order 4
on ${{\mathbb{S}}}^{4k+3}\subset\mathbb{C}^{2k+1}$ by generalizing Equation~
(\ref{E2.b}) to define
$$
T(\vec z)=(-\bar z_2,\bar z_1,\dots,-\bar z_{2k+2},\bar z_{2k+1})\,.
$$
Since $T(c\vec z)=\bar c\,T(\vec z)$,
$T$ descends to $\mathbb{CP}^{2k+1}$ and we can define
$$
\langle T\rangle(\langle\vec z\rangle)=\langle T\vec z\rangle\,.$$
Since $T^2(\vec z)=-\vec z$, $\langle T\rangle^2\langle\vec z\rangle=\langle\vec z\rangle$ and
$\langle T\rangle$ is idempotent. If we assume $T$
has a fixed point projectively, then $z_1=-c\bar z_2$ and $z_2=c\bar z_1=-|c|^2z_2$ for $c\ne0$.
Thus $z_1=z_2=0$ and we conclude
inductively $\langle\vec z\rangle=\langle\vec 0\rangle$ which is, of course not possible. We take, therefore,
$\mathbb{M}=\mathbb{CP}^{2k+1}/\mathbb{Z}_2$ as our exemplar.
Alternatively, we may
{identify ${{\mathbb{C}^{2k+2}}}$ with $\mathbb{H}^{k+1}$}. We then have $T(\vec x)=j\cdot\vec x$
so the quaternions make their appearance. The right action by the symplectic group acts
transitively on the unit sphere in $\mathbb{H}^{k+1}$ and commutes with left multiplication
by the quaternions. Consequently, $\mathbb{M}$ is a homogeneous space.

We investigate the distance functions on $\mathbb{S}^{4k+3}$ and on
$\mathbb{CP}^{2k+1}$. We take as basepoint the north pole $P=(1,0,\dots,0)$; the particular
point in question is irrelevant as we are dealing with homogeneous spaces.
If $\vec\xi\in P^\perp$ with $\|\vec\xi\|=1$,
then the curve
$$
\sigma(t):=\cos(t)P+\sin(t)\vec\xi\quad\text{for}\quad t\in(-\pi,\pi)
$$
is a unit speed
geodesic from $P$ with initial direction $\xi$.
Thus
\begin{eqnarray*}
&&d_{\mathbb{S}^{2k+1}}(P,\sigma(t))=t=\arccos(x^1)\quad\text{so}\quad
r_{\mathbb{S}^{2k+1},P}(\vec z)=\arccos(\Re(z^1))\\
&&r_{\mathbb{CP}^k,\langle P\rangle}\langle\vec z\rangle
=\min_{\lambda\in S^1}r_{\mathbb{S}^{2k+1}}(\lambda\vec z)=\arccos(|z^1|)\,.
\end{eqnarray*}
Let $\pi:\mathbb{CP}^{2k+1}\rightarrow\mathbb{M}:=\mathbb{CP}^{2k+1}/\mathbb{Z}_2$ be the natural projection. We then have
$$
r_{\mathbb{M},\pi\langle P\rangle}\pi(\langle\vec z\rangle)=\arccos(\min(|z^1|,|z^2|))\,.
$$
This shows that $\iota_{\mathbb{M},P}=\frac\pi4$ and the maximal domain of $r_{\mathbb{M}}$ is given by the open set
$$\{\langle t,z^2,\dots,z^k\rangle:|z^2|<t\}\subset\{(t,z^2,\dots,z^k):t^2+\|(z^2,\dots,z^k)\|^2=1\}\,.$$
The requirement that the first component is positive normalizes the point of projective space and the requirement that the action of $\mathbb{Z}_2$
is suitably normalized is reflected in the requirement $|z^2|<t$. We present a picture of the
slice obtained by taking $z^3=\dots=z^k=0$; the $t$ axis is up; it is a 45 degree cone, shown in
dark blue, about
the z-axis where the boundary circle is identified by a 90 degree rotation; it is $\mathbb{RP}^2$
topologically speaking. The full sphere is shown in gray but is not part of $\mathbb{M}$.
\smallbreak\vbox{\centerline{Picture 2.2}\smallbreak
\centerline{\includegraphics[height=3cm,keepaspectratio=true]{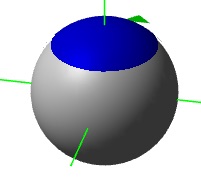}}}

Previously, we let ${\rm CP}^{2k+1}={\rm S}^{4k+3}/{\rm S}^1$. We now regard
$$\rm{CP}^{2k+1}=\{{\mathbb C}^{2k+2}-\{0\}\}/\{{\mathbb C}-\{0\}\}
$$
and introduce corresponding homogeneous coordinates $\langle\vec z\rangle$ for
$\vec z\in\mathbb{C}^{2k+2}-\{0\}$.
If $k=0$, then $\mathbb{CP}^1=\mathbb{S}^2=\mathbb{C}\cup\{\infty\}$ is the Riemann sphere and the map is
$\langle 1,w^2\rangle\rightarrow\langle- \bar w^2,1\rangle =\langle1,-\frac1{\bar w^2}\rangle$ away
from $\langle1,0\rangle$.
This is anti-holomorphic and hence reverses the orientation. More generally, we have
$$
\langle T\rangle\left\{\left\langle 1,w^2,\dots,w^{2\ell+2}\right\rangle\right\}
=\left\langle1,-\frac1{\bar w^2},\frac{\bar w^4}{\bar w^2},
-\frac{\bar w^3}{\bar w^2},\dots\right\rangle\quad\text{for}\quad w^2\ne0\,.
$$
This is an anti-holomorphic function of $2k+1$ complex variables
and hence reverses the orientation as expected.
\begin{remark}\rm Theorem~9.3.1 of Wolf~\cite{W2010} shows that any other
compact Riemannian manifold which is not simply-connected and which is modeled
on $\mathbb{CP}^{2k+1}$ for $k>0$ is in fact isometric to the example described above.
\end{remark}

\section{Flat space}\label{S3}If $\mathbb{M}$ is a simply-connected complete harmonic space with
scalar curvature $\tau=0$, then $\mathbb{M}$ is flat and
$\mathbb{M}=(\mathbb{R}^m,g_e)$ is isometric Euclidean space with the standard flat metric. This is a homogeneous space and $\iota_{\mathbb{M},P}=\infty$
for any $P$. We have
$$\Theta_{\mathbb{M},P}=r^{m-1},\quad\phi_1=r^{1-m},\quad
\phi_0=\left\{\begin{array}{ll}\log(r)&\text{if }m=2\\ (2-m)^{-1}r^{2-m}&\text{if }m>2\end{array}\right\}\,.$$

More generally, if $\mathbb{M}$ is a complete harmonic space with $\tau=0$,
then the universal cover of $\mathbb{M}$ is $\mathbb{R}^m$. Let $\Gamma:=\pi_1(\mathbb{M})$
and let $\pi:\mathbb{R}^m\rightarrow\mathbb{R}^m/\Gamma=\mathbb{M}$
be the universal cover map. If $\mathbb{M}$ is orientable,
then $\mathbb{M}=\mathbb{R}^k\times\{\mathbb{R}^{m-k}/\Gamma\}$ where
$\Gamma$ is a co-compact lattice in $\mathbb{R}^{m-k}$ that defines a torus of dimension $m-k$.
More generally, $\mathbb{M}$ is a generalized Klein bottle.
We illustrate the results of Theorem~\ref{T1.2} in Section~\ref{S3.1} by examining the square torus,
and in Section~\ref{S3.2}, by studying the open Mobius strip.
\subsection{The square torus}\label{S3.1}
Let $\mathbb{Z}^2$ be the integer lattice in $\mathbb{R}^2$ and let
$\mathbb{M}:=\mathbb{R}^2/\mathbb{Z}^2$; this is a homogeneous space
so without loss of generality we may assume $\tilde P=0$ and $P=0$. We then
have $\iota_{\mathbb{M},0}=\frac12$ and $\tilde{\mathcal{O}}_{\tilde0}=(-\frac12,\frac12)\times(-\frac12,\frac12)$ which we identify with $\mathcal{O}_0$ in $\mathbb{M}$.
The cut-locus is the image of the boundary of the closed
square $[-\frac12,\frac12]\times[-\frac12,\frac12]$ and is topologically a figure 8.
We picture the
situation below.
\smallbreak\vbox{\centerline{Picture 3.1}\smallbreak\centerline
{\includegraphics[height=5cm,keepaspectratio=true]{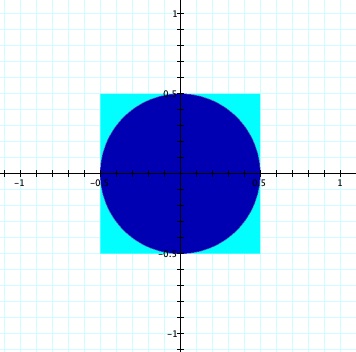}}}
\begin{remark}\label{R3.1}\rm
If $\mathcal{O}(\delta)=(-\delta,1-\delta)\times(-\delta,1-\delta)$ for $0<\delta<1$, then we have that
$\phi_0^\delta:=\log(x^2+y^2)$ is a harmonic real analytic function on $\mathcal{O}(\delta)$
which is radial near $0$ and which agrees with $\phi_0$ near the origin. It is radial on all of $\mathcal{O}(\delta)$ if and only if $\delta=\frac12$.
\end{remark}

\subsection{The open Klein bottle}\label{S3.2}
There is a unique real line bundle $\mathbb{M}$ over the circle which is not orientable. We can realize this
as $[-\frac12,\frac12]\times\mathbb{R}/\sim$ where we identify
the point $(-\frac12,y)$ with the point $(\frac12,-y)$. Alternatively, let
$T(x,y)=(x+1,-y)$ define an action of $\mathbb{Z}$ on $\mathbb{R}^2$.
Then we may identify $\mathbb{M}$ with
$\mathbb{R}^2/\mathbb{Z}$. Let $P=(x,y)\in\mathbb{R}^2$. The translations $(x,y)\rightarrow(x+a,y)$ and the reflection
$(x,y)\rightarrow(x,-y)$ commute with the action
of $\mathbb{Z}$ so $\mathbb{M}$ is homogeneous in the $x$ direction and also the reflection
$(x,y)\rightarrow(x,-y)$. Thus we may assume that $P=(0,a)$ for $a\ge0$. By Equation~(\ref{E1.c}),
\begin{equation}\label{E3.a}\begin{array}{l}
\iota_{\mathbb{M},(0,a)}=\frac12\min_{n\ne0}\|(0,a)-T^n(0,a)\|\\[0.05in]
\quad=\frac12\min\left\{\vphantom{\vrule height 10pt}\min_{k\ne0}\{\|(0,a)-(2k,a)\|\},\min_k\{\|(0,a)-(2k+1,-a)\|{{\}}}\right\}\\[0.05in]
\quad=\frac12\min\{2,\sqrt{1+4a^2}\}\,.
\end{array}\end{equation}
In particular, $\mathbb{M}$ is not homogeneous since $\iota_{\mathbb{M},P}$ is not constant.

We now determine the cut locus.
Suppose first $P=(0,0)$. Let $\mathcal{O}:=(-.5,.5)\times\mathbb{R}$. Then
$\pi$ is a diffeomorphism from $\mathcal{O}$ to an open neighborhood of the origin
in $\mathbb{M}$ and the argument given above for the torus shows
$r_{\mathbb{M},P}(x,y)=(x^2+y^2)^{1/2}$ on $\mathcal{O}$. The cut locus is
$\pi\{\pm\frac12\times\mathbb{R}\}$ and $r$ is not smooth there; $\mathcal{O}$ is
a maximal domain for the non-constant radial harmonic function $\phi_1$.
We refer to Picture 3.2 below.

On the other hand, if $a\ne0$, the situation is a bit different. We consider the
set of points $(x,y)$ so that $\|(x,y)-(0,a)\|<\|T^n(x,y)-(0,a)\|$ for all $n\ne0$,
i.e. $(x,y)$ is the closest point to $(0,a)$ among all possible pre-images of $(x,y)$. We compute:
\begin{eqnarray*}
&&\|T^n(x,y)-(0,a)\|^2-\|(x,y)-(0,a)\|^2\\
&=&\left\{\begin{array}{ll}n^2+2nx&\text{if }n\text{ is even}\\
n^2+2nx+4ay&\text{if }n\text{ is odd}\end{array}\right\}\,.
\end{eqnarray*}
If $n=\pm2$, then we obtain $4\pm4x>0$ and hence $|x|<1$. We impose this condition
by restricting to $\mathcal{O}=(-1,1)\times\mathbb{R}$. We then have $n^2+2nx>0$ if $n\ne0$.
And, furthermore, if $n$ is odd, then $n^2+2nx+4ay$ attains its minimum when $n=\pm1$. Thus
we must impose the conditions:
$$
1+2x+4ay>0\text{ and }1-2x+4ay>0\,.
$$
Thus the fundamental region on which $r$ is smooth is the region
$$
-1<x<1,\quad 1+2x+4ay>0,\quad 1-2x+4ay>0\,.
$$

Suppose $a>0$. To create the final region, the line from $(-1,\frac14a)$ to $(-1,\infty)$ is glued to the line from $(1,\frac1a)$ to $(1,\infty)$ using $T^2$,
and the directed line segment from $(-1,\frac14a)$ to $(0,-\frac14a)$ is glued
to the line segment $(0,-\frac14a)$ to $(1,\frac14a)$ using $T$. The pattern is similar if $a<0$.
As $a\rightarrow0$, the two lines $1\pm2x+4ay=0$ converge to the lines $x=\pm\frac12$.
The closest points to the center on the 4 line segments are at $(a,\pm1)$ and $(0,\pm\frac12)$. Thus,
as noted in Equation~(\ref{E3.a}), the injectivity radius
is $\iota_{\mathbb{M},P}=\min(1,\sqrt{a^2+\frac14})$; $(0,0)$ has the smallest injectivity radius of $\frac12$ and the points $(0,a)$ for $a\ge\frac{\sqrt 3}2$
have the largest injectivity radius of 1.

\smallbreak\vbox{\centerline{Picture 3.2}\smallbreak\centerline{\begin{tabular}{cccc}
\includegraphics[height=3.5cm,keepaspectratio=true]{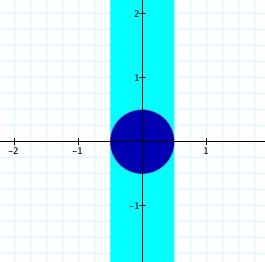}&
\includegraphics[height=3.5cm,keepaspectratio=true]{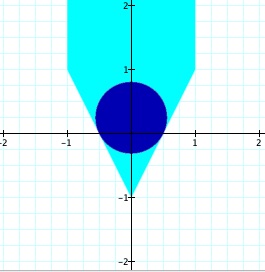}&
\includegraphics[height=3.5cm,keepaspectratio=true]{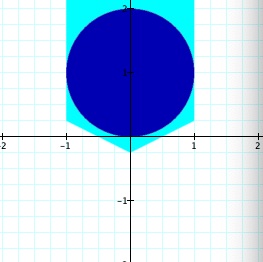}\\
$\iota_{\mathbb{M},(0,0)}=\frac12$&$\iota_{\mathbb{M},(0,\frac14)}\approx.56$&$\iota_{\mathbb{M},(0,1)}=1$
\end{tabular}}}

It is worth noting that there are other maximal domains for $\phi_1$ other than those provided
by removing the cut locus. Let $P=(0,0)$ and let $\mathcal{O}_2:=(-.25,.75)\times\mathbb{R}$.
Then $\Phi_1=\log(x^2+y^2)$ is harmonic on $\mathcal{O}_2$. More generally, if
$\mathcal{U}$ is any connected open subset of $\mathbb{R}^2$ containing $(0,a)$ on which $\pi$ is 1-1,
then $\pi(\mathcal{U})$ is a domain on which $\Phi_1:=\log(x^2+(y-a)^2)$ is well defined and
harmonic but it will not be radial on this domain if the domain contains any cut points.

\section{Negative curvature}\label{S4}
\subsection{Simply connected examples} The solutions $\phi_0$ and $\phi_1$ of
Table~4.1 must be adjusted to replace the corresponding trignometric functions by their
hyperbolic analogues with appropriate sign changes to obtain the corresponding solutions for the negatively curved
rank 1 symmetric spaces; they are defined on all of $\mathbb{M}-\{P\}$. We obtain:
\goodbreak\smallbreak\vbox{\centerline{Table 4.1}
\smallbreak\centerline{$\begin{array}{| l | l | l |}\noalign{\hrule}\mathbb{M}&\phi_1&\phi_0\\\noalign{\hrule}
\vphantom{\vrule height 12pt}\widetilde{\mathbb{S}}^2&\frac1{\sinh(r)}&\log \left(\tanh \left(\frac{r}{2}\right)\right)\\\noalign{\hrule}
\vphantom{\vrule height 12pt}\widetilde{\mathbb{S}}^3&\frac1{\sinh^2(r)}&-\coth (r)\\\noalign{\hrule}
\vphantom{\vrule height 12pt}\widetilde{\mathbb{S}}^4&\frac1{\sinh^3(r)}&-\frac{1}{8} \text{csch}^2\left(\frac{r}{2}\right)-\frac{1}{8} \text{sech}^2\left(\frac{r}{2}\right)-\frac{1}{2} \log \left(\tanh \left(\frac{r}{2}\right)\right) \\ \noalign{\hrule}
\vphantom{\vrule height 12pt}\widetilde{\mathbb{S}}^5&\frac1{\sinh^4(r)}&\frac23 \coth (r)-\frac{1}{3} \coth (r) \text{csch}^2(r) \\
\noalign{\hrule}
\widetilde{\mathbb{CP}}^2&\frac1{\sinh(r)^3\cosh(r)}&
-\frac{1}{2} \text{csch}^2(r)-\log (\tanh (r))\\\noalign{\hrule}
\widetilde{\mathbb{CP}}^3&\frac1{\sinh(r)^5\cosh(r)}&
-\frac{1}{4} \text{csch}^4(r)+\frac12\text{csch}^2(r)+\log (\tanh (r))\\ \noalign{\hrule}
\widetilde{\mathbb{CP}}^4&\frac1{\sinh(r)^7\cosh(r)}&
-\frac{1}{6} \text{csch}^6(r)+\frac14\text{csch}^4(r)-\frac12\text{csch}^2(r)-\log (\tanh (r))\\\noalign{\hrule}
\widetilde{\mathbb{HP}}^2&\frac1{\sinh(r)^7\cosh(r)^3}&
\frac{1}{2} \left(-\frac{1}{3} \text{csch}^6(r)+\text{csch}^4(r)-3 \text{csch}^2(r)-\text{sech}^2(r)\right.\\
&&\left.\quad-8 \log (\tanh (r))\right)
\\ \noalign{\hrule}
\widetilde{\mathbb{HP}}^3&\frac1{\sinh(r)^{11}\cosh(r)^3}&
-\frac{1}{10} \text{csch}^{10}(r)+\frac14\text{csch}^8(r)-\frac12\text{csch}^6(r)+\text{csch}^4(r)\\
&&-\frac52 \text{csch}^2(r)
-\frac12\text{sech}^2(r)
-6 \log (\tanh (r))
\\ \noalign{\hrule}
\widetilde{\mathbb{HP}}^4&\frac1{\sinh(r)^{15}\cosh(r)^3}&
\frac{1}{2} \left(-\frac{1}{7} \text{csch}^{14}(r)+\frac13\text{csch}^{12}(r)-\frac35 \text{csch}^{10}(r)
+\text{csch}^8(r)\right.\\
&&\left.\quad-\frac53 \text{csch}^6(r)+3 \text{csch}^4(r)
-7 \text{csch}^2(r)-\text{sech}^2(r)\right.\\
&&\left.\quad-16 \log (\tanh (r))\right)
\\\noalign{\hrule}\noalign{\goodbreak\hrule}
\widetilde{\mathbb{OP}}^2&\frac1{\sinh(r)^{15}\cosh(r)^7}&
-\frac{1}{14} \text{csch}^{14}(r)+\frac13\text{csch}^{12}(r)-\text{csch}^{10}(r)+\frac52 \text{csch}^8(r)\\
&&-\frac{35}6 \text{csch}^6(r)
+14 \text{csch}^4(r)-42 \text{csch}^2(r)-\frac16\text{sech}^6(r)\\
&&-2 \text{sech}^4(r)-18 \text{sech}^2(r)-120 \log (\tanh (r))
\\ \noalign{\hrule}\end{array}$}}

\subsection{The proof of Theorem~\ref{T1.4}}\label{S4.2}
Let $\mathbb{M}$ be a compact Riemannian manifold of dimension $m=4\ell$
which is modeled on a rank 1 symmetric space $\mathbb{M}_-$ of negative curvature. Let $\mathbb{M}_+$
be the dual rank 1 symmetric space of positive curvature.
The argument of Section~\ref{S2} using the Chern-Gauss-Bonnet Formula to obtain information on
$|\Gamma|$ in the positive curvature setting can be turned around to obtain
information on $\operatorname{Vol}(\mathbb{M})$.
Let $\operatorname{Pf}_\ell$ be the Pfaffian; this is a polynomial of degree $\ell$ in the curvature tensor
so $\operatorname{Pf}_\ell(\mathbb{M})=\operatorname{Pf}_\ell(\mathbb{M}_-)=(-1)^\ell\operatorname{Pf}_\ell(\mathbb{M}_+)$;
these are constant since the spaces in question are locally homogeneous. The Chern-Gauss-Bonnet Formula
then shows
\begin{eqnarray*}
\chi(M_+)&=&\operatorname{Pf}_\ell(\mathbb{M}_+)\cdot\operatorname{vol}(\mathbb{M}_+),\\
\chi(M)&=&\operatorname{Pf}_\ell(\mathbb{M})\cdot\operatorname{vol}(\mathbb{M})
=(-1)^\ell\operatorname{Pf}_\ell(\mathbb{M}_+)\cdot\operatorname{vol}(\mathbb{M}_+)\cdot
\frac{\operatorname{vol}(\mathbb{M})}{\operatorname{vol}(\mathbb{M}_+)}\\
&=&(-1)^\ell\chi(M_+)\cdot
\frac{\operatorname{vol}(\mathbb{M})}{\operatorname{vol}(\mathbb{M}_+)}\,.
\end{eqnarray*}
Since $\chi(M_+)>0$, this implies $\chi(M)\ne0$. Consequently
$$
1\le|\chi(M)|=\chi(M_+)\frac{\operatorname{vol}(\mathbb{M})}{\operatorname{vol}(\mathbb{M}_+)}
\quad\text{so}\quad\operatorname{vol}(\mathbb{M})
\ge\frac{\operatorname{vol}(\mathbb{M}_+)}{\chi(M_+)}\,.
$$
Theorem~\ref{T1.4} now follows since, by Lemma~\ref{L2.1},
\medbreak\hfill$\chi({S}^m)=2$, $\chi({CP}^k)=k+1$,
$\chi({HP}^k)=k+1$, and $\chi({OP}^2)=3$.\hfill\vphantom{.}
\qed
\subsection{The proof of Theorem~\ref{T1.5}}\label{S4.3} Let $\mathbb{M}$ and $\mathbb{M}_{\pm}$
be as in Section~\ref{S4.2}. Suppose that $m=4j$ is divisible by 4.
The estimates of Theorem~\ref{T1.4} arise from the Chern-Gauss-Bonnet Formula and are not
optimal in certain settings. By using the Hirzebruch Signature Formula, we can obtain better
estimates. Let
$$
\mathbb{M}_-\in\{\widetilde{\mathbb{CP}}^{2k},\widetilde{\mathbb{HP}}^{4k},\widetilde{\mathbb{OP}}^2\}
\quad\text{so}\quad
\mathbb{M}_+\in\{{\mathbb{CP}}^{2k},{\mathbb{HP}}^{4k},{\mathbb{OP}}^2\}\,.
$$

Suppose first that $\mathbb{M}$ is orientable.
The Hirzebruch polynomial $L_j$ is a constant
 multiple of the oriented volume form. Since $L_j$ is quadratic in the curvature tensor, we have that
 $L_j(\mathbb{M})=L_j(\mathbb{M}_-)=L_j(\mathbb{M}_+)$
 where we adust the orientations to ensure this equality. By Lemma~\ref{L2.1},
 $\operatorname{sign}(M_+)=1$,
 Therefore,
 \begin{eqnarray*}
\operatorname{sign}(M_+)&=&L_j(\mathbb{M})\cdot\operatorname{vol}(\mathbb{M}),\\
\operatorname{sign}(M)&=&L_j(\mathbb{M})\cdot\operatorname{vol}(\mathbb{M})
=L_j(\mathbb{M})\cdot\operatorname{vol}(\mathbb{M}_+)\cdot
\frac{\operatorname{vol}(\mathbb{M})}{\operatorname{vol}(\mathbb{M}_+)}\\
&=&\operatorname{sign}(M_+)\cdot\frac{\operatorname{vol}(\mathbb{M})}{\operatorname{vol}(\mathbb{M}_+)}
=\frac{\operatorname{vol}(\mathbb{M})}{\operatorname{vol}(\mathbb{M}_+)}\,.
\end{eqnarray*}
This shows $\operatorname{sign}(M)$ is positive and hence
$\operatorname{sign}(M)\ge1$
so $\operatorname{vol}(\mathbb{M})\ge\operatorname{vol}(\mathbb{M}_+)$.
If $\mathbb{M}$ is not orientable, we let $\mathbb{M}_0$
be the orientable double cover and estimate
$\frac12\operatorname{vol}(\mathbb{M})=
\operatorname{vol}(\mathbb{M}_0)\ge\operatorname{vol}(\mathbb{M}_+)$.\qed

\section*{Dedication}
On 11 March 2004, 10 bombs exploded on 4 trains near the Atocha Station
in Madrid killing 191 and injuring more than 1800; 18 Islamic fundamentalists and
3 Spanish accomplices were convicted of the bombings which was one of Europe's deadliest
terrorist attacks in the years since World War II. Subsequently, one of us (Gilkey) and his coauthors
dedicated a paper \cite{DFGG04} writing ``En memoria de todas las
v\'\i ctimas inocentes. Todos \'\i bamos en ese tre{n}.
(In memory of all these innocent victims. We were all on that train)''.
This paper is being written during one of the worst outbreaks of war in Europe since World War II.
We dedicate this paper, writing in
a similar vein to show our solidarity with the innocent victims in Ukraine, that:
  Mи всі в Україні (``we are all in Ukraine'').

\section*{Research support}
The research of P. Gilkey was partially supported by grant PID2020-114474GB-I0 (Spain).
The research of J. H. Park was partially supported by the National Research Foundation of Korea (NRF) grant
funded by the Korea government (MSIT) (NRF-2019R1A2C1083957).
Helpful suggestions and comments were provided by our colleague and friend
E. Garc\'{i}a-R\'{i}o.

\end{document}